\documentclass[11pt]{amsart}
\usepackage{amssymb,amsmath}

\newcommand{\Qqq}{{\mathbb Q}}
\newcommand{\Rrr}{{\mathbb R}}

\newcommand{\bin}{\operatorname{Bin}}

\newtheorem{theorem}{Theorem}[section]
\newtheorem{proposition}[theorem]{Proposition}
\newtheorem{lemma}[theorem]{Lemma}
\newtheorem{definition}[theorem]{Definition}
\newtheorem{notation}[theorem]{Notation}
\newtheorem{corollary}[theorem]{Corollary}
\newtheorem{example}[theorem]{Example}
\newtheorem{remark}[theorem]{Remark}

\newcommand{\hz}{\widehat{0}}
\newcommand{\ho}{\widehat{1}}

\newcommand{\tensor}{\otimes}
\newcommand{\rightarrowdots}{\rightarrow \cdots\rightarrow}
\newcommand{\cupdots}{\cup \cdots \cup}

\newcommand{\ab}{\av\bv}
\newcommand{\av}{{\bf a}}
\newcommand{\bv}{{\bf b}}
\newcommand{\cd}{\cv\dv}
\newcommand{\cv}{{\bf c}}
\newcommand{\dv}{{\bf d}}
\newcommand{\tv}{{\bf t}}

\newcommand{\walks}[2]{#1 \langle #2\rangle}
\newcommand{\wt}{\operatorname{wt}}

\makeatletter
\@addtoreset{equation}{section}
\makeatother

\newcommand{\trace}{\operatorname{trace}}

\begin{document}

\title{Level Eulerian Posets}

\author{Richard EHRENBORG \and G\'abor HETYEI \and Margaret READDY}

\address{Department of Mathematics, University of Kentucky, Lexington,
  KY 40506-0027 .\hfill\break WWW: \tt http://www.ms.uky.edu/\~{}jrge/.}

\address{Department of Mathematics and Statistics,
  UNC-Charlotte, Charlotte NC 28223-0001.\hfill\break
WWW: \tt http://www.math.uncc.edu/\~{}ghetyei/.}

\address{Department of Mathematics, University of Kentucky, Lexington,
  KY 40506-0027 .\hfill\break
WWW: \tt http://www.ms.uky.edu/\~{}readdy/.}



\begin{abstract}
The notion of level posets is introduced.
This class of infinite posets has the property that between every two
adjacent ranks the same bipartite graph occurs.
When the adjacency matrix is indecomposable, 
we determine the length of the longest
interval one needs to check to verify 
Eulerianness.
Furthermore, we show that every level Eulerian poset
associated to an indecomposable matrix has even order.
A condition for verifying shellability is introduced
and is automated using the algebra of walks.
Applying the Skolem--Mahler--Lech theorem,
the $\ab$-series of a level poset
is shown to be a rational generating function in
the non-commutative variables $\av$ and $\bv$.
In the case the poset is also Eulerian,
the analogous result holds for the $\cd$-series.
Using coalgebraic techniques a method is developed
to recognize the $\cd$-series matrix of a level Eulerian poset.
\end{abstract}

\maketitle

\section{Introduction}

It is the instinct of every mathematician that whenever an infinite
object is defined in terms of a finite object one should be able to
describe various apparently infinite properties of the infinite object
in finitely-many terms. For example, when one considers infinite random
walks on a finite digraph it is satisfying to be able to describe the
asymptotic properties of these random walks in terms of the
finitely-many eigenvalues of the adjacency matrix.

In this paper we consider infinite partially ordered
sets (posets) associated to finite
directed graphs. The level poset of a graph that we introduce is an
infinite voltage graph closely related to the finite voltage graphs
studied by Gross and Tucker. It is a natural question to consider whether
a level poset has Eulerian intervals,
that is, every non-singleton interval satisfies the Euler-Poincar\'e relation.
One method to form Eulerian posets is via the doubling operations.
This corresponds to a standard trick widely used
in the study of network flows.
We extend these operations to level posets.

We look at questions that are often asked in the study of Eulerian
posets: verifying Eulerianness,
finding sufficient conditions which imply the order complex
is shellable and describing the flag numbers.  Usually these questions are
aimed at a specific family of finite posets and
explicit answers are given. Here we instead look at infinitely-many intervals
defined by a single finite directed graph.

When the underlying graph of a level poset is strongly connected,
we show that it is enough to verify the Eulerian condition for intervals
up to a certain rank. This bound is linear in terms of
the two parameters period and index of the level poset.
Furthermore, for these Eulerian level posets we also
obtain that their order must be even.
The order two Eulerian poset is the classical butterfly poset.
See Example~\ref{example_order_4}
for an order $4$ example.

To show that a level poset has shellable intervals
we introduce the vertex shelling order condition.
This condition is an instance of
Kozlov's $CC$-labelings
and it implies shellability.
Furthermore,
we prove it is enough to verify this condition
for intervals whose length is bounded by the sum of
period and the index.
This is still a large task. However, we automate
it using the algebra of walks, reducing the problem of
computing powers of a certain matrix modulo an ideal.
See Example~\ref{example_order_4_shelling}
for such a calculation.
In this example we conclude that the
order complexes of the intervals are
not just homotopic equivalent
to spheres, but homeomorphic to them.

The $\cd$-index is an invariant encoding
the flag $f$-vector of an Eulerian poset which
removes all linear relations among the flag $f$-vector entries.
It is a non-commutative homogeneous polynomial 
in the two variables $\cv$ and $\dv$. For level Eulerian posets
there are infinitely-many intervals. We capture this information
by summing all the $\cd$-indicies. This gives a non-commutative
formal power series which we call the $\cd$-series. We show that the
$\cd$-series is a rational non-commutative generating function.
See Theorem~\ref{theorem_cd_rational}.

Recall that the infinite butterfly poset has the property
that the $\cd$-index of any length $m+1$ interval 
equals $\cv^{m}$.
In our order $4$ example
of a level Eulerian poset,
there are intervals of length $m+1$ whose
$\cd$-index is the sum of every degree $m$ $\cd$-monomial.
See Corollary~\ref{corollary_sum_of_monomials}.

In the concluding remarks we end with some open questions.

\section{Preliminaries}

\subsection{Graded, Eulerian and half-Eulerian posets}

A partially ordered set $P$ is {\em graded} if it has a unique minimum element
$\hz$, a unique maximum element $\ho$ and a rank function $\rho:P\rightarrow
{\mathbb N}$  such that $\rho(\hz)=0$ and for every cover relation
$x \prec y$
we have $\rho(y)-\rho(x)=1$. The rank of $\ho$ is called the 
{\em rank} of the
poset.
For two elements
$x \leq y$ in $P$ define
the rank difference
$\rho(x,y)$ by $\rho(y) - \rho(x)$.
Given a graded poset $P$ of rank $n+1$ and a subset
$S\subseteq \{1,\ldots,n\}$, define the {\em
$S$-rank selected subposet of $P$} to be the poset
$P_{S} = \{ x \in P\::\: \rho(x) \in S\} \cup \{\hz,\ho\}.$
The {\em flag $f$-vector}
$(f_{S}(P)\: :\: S \subseteq \{1, \ldots, n\})$
of $P$ is
the $2^n$-dimensional vector whose entry
$f_S(P)$ is the number of maximal chains in $P_S$. For further details
about graded posets, see Stanley~\cite{Stanley_EC_1}. 

A graded partially ordered set $P$ is {\em Eulerian} if every interval
$[x,y]$ in $P$ of rank at least $1$
satisfies $\sum_{x \leq z \leq y} (-1)^{\rho(z)}=0$.
Equivalently, the M\"obius function $\mu$ of the poset $P$
satisfies $\mu(x,y) = (-1)^{\rho(x,y)}$.
Classical examples of Eulerian posets include the face lattices
of polytopes and the Bruhat order of a Coxeter group.

The {\em horizontal double} $D_{\leftrightarrow}(P)$
of a graded poset $P$ is obtained by
replacing each element $x\in P - \{\hz,\ho\}$ by two copies $x_{1}$ and
$x_{2}$ and preserving the partial order of the original poset $P$, that
is, we set $x_{i}<y_{j}$ in $D_{\leftrightarrow}(P)$
if and only if $x<y$ holds in $P$.
Following Bayer and Hetyei~\cite{Bayer_Hetyei_E,Bayer_Hetyei_G},
we call a graded poset $P$ {\em half-Eulerian} if its horizontal double
is Eulerian. The following lemma appears
in~\cite[Proposition~2.2]{Bayer_Hetyei_E}. 
\begin{lemma}[Bayer--Hetyei]
\label{lemma_h_double}
A graded partially ordered set $P$ is half-Eulerian if and only if for every
non-singleton interval $[x,y]$ of $P$
$$
    \sum_{x < z < y}
        (-1)^{\rho(x,z) -1}
  =
    \left\{ \begin{array}{c l}
             1 & \text{ if $\rho(x,y)$ is even,} \\
             0 & \text{ if $\rho(x,y)$ is odd.}
            \end{array} \right.
$$
\end{lemma}
As noted in~\cite[Section~4]{Bayer_Hetyei_G}, every graded poset $P$
gives rise to a half-Eulerian poset via the
``vertical doubling'' operation.

\begin{definition}
Given a graded poset $P$,
the {\em vertical double} of $P$ is the set
 $D_{\updownarrow}(P)$ obtained by replacing each $x\in P- \{\hz,\ho\}$ by two
copies $x_{1}$ and $x_{2}$, with $u<_{D_{\updownarrow}(P)} v$ in $Q$
exactly when one of the following conditions hold:
\begin{itemize}
\item[(i)] $u=\hat{0}$, $v\in P-\{\hat{0}\}$;
\item[(ii)] $u\in P-\{\hat{1}\}$, $v=\hat{1}$;
\item[(iii)] $u=x_{1}$ and $v=x_{2}$ for some $x\in
  P-\{\hat{0},\hat{1}\}$; or
\item[(iv)] $u=x_{i}$ and $v=y_{j}$ for some $x, y\in
  P-\{\hat{0},\hat{1}\}$, with $x<_P y$.
\end{itemize}
\end{definition}

\begin{lemma}[Bayer--Hetyei]
\label{lemma_v_double}
For a graded poset $P$,
the vertical double
$D_{\updownarrow}(P)$ is a half-Eulerian poset.
 \end{lemma}

\subsection{Shelling the order complex of a graded poset}

Recall a simplicial complex $\Delta$ is a family of subsets
({\em faces}) of a finite vertex set $V$ satisfying $\{v\} \in \Delta$
for all $v \in V$ and if $\sigma \in \Delta$ and $\tau\subseteq \sigma$
then $\tau\in \Delta$. Maximal faces are called {\em facets}.
In this paper we will only consider order
complexes of graded posets. The {\em order complex $\Delta(P)$}
of a graded poset $P$ is the simplicial complex with vertex set
$P - \{\hz,\ho\}$ whose faces are the chains
contained in $ P- \{\hz,\ho\}$, that is,
$$   \Delta(P)
   =
     \{\{x_{1},x_{2}, \ldots, x_{k}\}
          \:\: : \:\:
         \hz < x_{1} < x_{2} < \cdots < x_{k} < \ho\}   . $$

A simplicial complex is {\em pure} if every facet has the same dimension.
For a graded poset $P$ of rank $n+1$, the order complex 
$\Delta(P)$
is pure of dimension $n-1$. A pure simplicial complex $\Delta$ is
{\em shellable} if there is an ordering $F_{1},F_{2},\ldots,F_{t}$ of its
facets such that for every $k\in \{2,\ldots,t\}$
the collection of faces of $F_{k}$ contained in some earlier $F_{i}$ 
is itself a
pure simplicial complex of dimension $\dim(\Delta)-1$. Equivalently,
there exists a face $R(F_{k})$ of $F_{k}$, called the facet restriction,
 not contained in any earlier
facet such that every face $\sigma\subseteq F_{k}$ not contained in any
earlier $F_{i}$ contains $R(F_{k})$.
A complex being shellable implies
it is homotopy equivalent to a wedge of spheres of the same
dimension as the complex.  For further details,
we refer the reader to the articles of Bj\"orner
and Wachs~\cite{Bjorner,Wachs}.

A shelling of the order complex of a graded poset is usually found by
labeling the cover relations in the maximal chains
of $P$. The first such labelings were the {\em
  $CL$-labelings} introduced by Bj\"orner and
Wachs~\cite{Bjorner_Wachs_0,Bjorner_Wachs_1}. In this paper we will
consider a special example of Kozlov's {\em
  $CC$-labelings}~\cite{Kozlov}, which were rediscovered
independently by Hersh and Kleinberg (see the Introduction
of~\cite{Babson_Hersh}). 
 
\subsection{Periodicity of nonnegative matrices}

We will need a few facts regarding sufficiently high powers
of nonnegative square matrices. Unless noted otherwise, all statements cited
in this subsection may be found in the monograph of Sachkov and
Tarakanov~\cite[Chapter~6]{Sachkov_Tarakanov}.

The {\em underlying digraph $\Gamma(A)$} of a square matrix
$A=(a_{i,j})_{1\leq i,j\leq n}$ with nonnegative entries is the directed
graph on the vertex set $\{1,2,\ldots,n\}$ with $(i,j)$ being an edge if
and only if $a_{i,j}>0$. 
Here and in the rest of the paper we use the notation 
$A^{k}=(a^{(k)}_{i,j})_{1\leq i,j\leq n}$. 
Given a vertex $i$ such that there is a
directed walk of positive length from $i$ to $i$, the 
{\em period $d(i)$} 
of the vertex $i$ is the greatest common divisor of all
positive integers $k$ satisfying $a^{(k)}_{i,i}>0$.  
The period is constant on
strong components of~$\Gamma(A)$.  

\begin{lemma}
\label{lemma_period_constant}
If the vertices $i\neq j$ of $\Gamma(A)$ belong to the same strong
component then $d(i)=d(j)$.
\end{lemma}
The matrix $A$ is {\em indecomposable} or {\em
  irreducible} if for any $i,j\in \{1,2,\ldots,n\}$ there is a $t>0$ such
that the $(i,j)$ entry of $A^t$ is positive.
It is easy to
see that $A$ is indecomposable if and only if its
underlying digraph is strongly connected, that is, for any pair of vertices
$i$ and $j$ there is a directed walk from $i$ to~$j$. As
a consequence of Lemma~\ref{lemma_period_constant}
all vertices of the underlying
graph of an indecomposable matrix~$A$ have the same period. We call this
number the {\em period of the indecomposable matrix $A$}. Given an 
$n \times n$
indecomposable matrix of period $d$, for each $i=1,2,\ldots,n$
there exists an integer $t_{0}(i)$ such that $a^{(k d)}_{i,i}>0$ holds for
all $k\geq t_{0}(i)$. Using
this observation is easy to show the following
theorem. See~\cite[Theorem~6.2.2 and Lemma~6.2.3]{Sachkov_Tarakanov}. 
\begin{theorem}
\label{theorem_indecomposable-block-matrix}
Let $A$ be an $n\times n$ indecomposable nonnegative matrix of period
$d$. If $i$ is a fixed vertex of the digraph $\Gamma(A)$ then for any
other vertex $j$ there is a unique integer $r_{j}$
such that $0 \leq r_{j} \leq d-1$
and the following two statements hold:
\begin{itemize}
\item[(1)] $a^{(s)}_{i,j}>0$ implies $s \equiv r_{j} \bmod d$,
\item[(2)] there is a positive $t(j)$ such that
$a^{(k d + r_{j})}_{i,j}>0$ for
all $k \geq t(j)$.
\end{itemize}
Setting $j\in C_r$ if and only if $r_{j}=r$ provides a partitioning
$\{1,2,\ldots,n\}=\biguplus_{q=0}^{d-1} C_{q-1}$. Replacing $i$ with an
arbitrary fixed vertex results in the same ordered list of subclasses,
up to a cyclic rotation of the indices.
\end{theorem}
Ordering the elements of the set $\{1,2,\ldots,n\}$ in such a way that
the elements of each block $C_q$ form a consecutive sublist results in
a block matrix of the form
\begin{equation}
\label{equation_block-matrix}
A =
\begin{pmatrix}
0 &Q_{0,1} & 0      & \cdots & 0\\
0 & 0     & Q_{1,2} & \cdots & 0\\
\vdots & \vdots & \vdots & \ddots & \vdots \\
0 & 0     & 0      & \cdots & Q_{d-2,d-1}\\
Q_{d-1,0} & 0     & 0 & \cdots & 0\\
\end{pmatrix} ,
\end{equation}
where $Q_{q,q+1}$ occupies the rows
indexed with $C_q$ and the columns indexed by $C_{q+1}$ (here we set
$C_{d} = C_{0}$). The block matrix in~\eqref{equation_block-matrix} is
the {\em canonical form} of the indecomposable matrix $A$ with period
$d$.  By Theorem~\ref{theorem_indecomposable-block-matrix} there
is a $t>0$ such that the canonical form of $A^{t d + 1}$ is similar to the
one given in~\eqref{equation_block-matrix} with the additional property that
all entries in the blocks $Q_{q,q+1}$ are strictly positive.

An $n\times n$  matrix $A$ with nonnegative entries is {\em primitive}
if there is a $\gamma>0$ such that all entries of~$A^{\gamma}$ are
positive. The smallest $\gamma$ with the above property is the {\em
  exponent} of the primitive matrix $A$.
It is straightforward to see that a
nonnegative matrix is primitive if and only if it is indecomposable and
{\em aperiodic}, i.e., its period $d$ equals $1$. There is a quadratic
upper bound on the exponent of a primitive matrix due to Holladay and
Varga~\cite{Holladay_Varga}. See \cite[Theorem~6.2.10]{Sachkov_Tarakanov}.
\begin{theorem}[Holladay--Varga]
\label{theorem_exp-primitive}
The exponent $\gamma$ of an $n\times n$ primitive matrix  satisfies
$$
\gamma \leq n^{2} - 2n + 2.
$$
\end{theorem}

The matrix operator we are about to introduce will be frequently used
in our paper and makes also stating the next few results easier.
\begin{definition}
Given a matrix $A$ with nonnegative entries, let
$\bin(A)$ denote the binary matrix formed by replacing each
nonzero entry of $A$ with $1$.  We call the resulting matrix the {\em binary
  reduction} of $A$.  
\end{definition}
Thus a matrix $A$ is primitive if there exists a power
$k$ such that $\bin(A^{k}) = J$, where
$J$ is the matrix consisting of all $1$'s.

A generalization of Theorem~\ref{theorem_indecomposable-block-matrix}
may be found in the work of Heap and Lynn~\cite{Heap_Lynn}.
See also~\cite{Ptak,Ptak_Sedlacek,Rosenblatt}.
\begin{theorem}
\label{theorem_decomposable-block-matrix}
Given any non-negative square matrix $A$ there exists
integers $d$ and $\gamma$ such that
$\bin(A^{t+d}) = \bin(A^{t})$ for all $t \geq \gamma$.
\end{theorem}
The smallest integer $\gamma$ such that
for all $t \geq \gamma$ we have
$\bin(A^{t+d}) = \bin(A^{t})$
is known as the {\em index} of the matrix.
For a primitive matrix this is the exponent.
For estimates on the period $d$ and
the index~$\gamma$, we refer the reader to~\cite{Heap_Lynn}.
Here we only wish to emphasize the following 
immediate generalization of Theorem~\ref{theorem_exp-primitive}.
See~\cite[Equation~(1.4)]{Heap_Lynn}.  
\begin{theorem}
\label{theorem_exp-indecomposable}
Let $A$ be an $n\times n$ indecomposable matrix with nonnegative entries
having period $d$.
An upper bound for the index $\gamma$ of $A$
is $\gamma \leq (q^2-2q+2)d+2r$,
where $n = q d + r$ with $0 \leq r < d-1$.  
\end{theorem}

Let $\Gamma$ be any digraph on a vertex set $V$ such that its edge set
$E$ is a subset of $V\times V$, i.e., $\Gamma$ may have loops but no
multiple edges. Recall the {\em adjacency matrix} $A$ of $\Gamma$ is a
$|V|\times |V|$ matrix whose rows and columns are indexed by the
vertices. The entry in the row indexed by $u\in V$ and in the column
$v\in V$ is $1$ if $(u,v)\in E$, and zero otherwise. Clearly $\Gamma$
is the underlying graph of its adjacency matrix. We may extend
the above notions of period and aperiodicity from matrices to digraphs
by defining the period of a digraph to be the period of its adjacency
matrix.  See for instance~\cite{Perrin_Schutzenberger}. In particular, a
directed graph is aperiodic if and only if it is 
strongly connected and there is no $k>1$ that divides the length of
every directed cycle.

\section{Level posets}

\begin{definition}
A partially ordered set $P$ is a {\em level poset} if the set of its
elements is of the form $V\times {\mathbb Z}$ for some finite nonempty
set $V$, the projection
onto the second coordinate is a rank function, and for any $u,v\in V$
and $i\in {\mathbb Z}$ we have $(u,i)<(v,i+1)$ if and only if
$(u,0)<(v,1)$ holds.  
\end{definition}

\begin{figure}[t]
\setlength{\unitlength}{1.2mm}
\begin{center}
\begin{picture}(40,50)(0,0)
\put(0,25){
$\displaystyle
          \begin{pmatrix}
              1 & 1 & 1 & 0 \\
              1 & 0 & 1 & 1 \\
              0 & 1 & 0 & 1 \\
              1 & 0 & 1 & 1
          \end{pmatrix}
$}
\end{picture}
\hspace*{25 mm}
\begin{picture}(40,50)(-5,-10)

\multiput(0,0)(0,10){4}{
  \multiput(0,0)(10,0){4}{
    \put(0,0){\circle*{1.5}}
  }
}

\multiput(0,0)(0,10){3}{
  \put(0,0){\line(0,1){10}}
  \put(0,0){\line(1,1){10}}
  \put(10,0){\line(-1,1){10}}
  \put(10,0){\line(1,1){10}}
  \put(20,0){\line(-1,1){10}}
  \put(20,0){\line(1,1){10}}
  \put(30,0){\line(-1,1){10}}
  \put(30,0){\line(0,1){10}}

  \put(0,0){\line(2,1){20}}
  \put(10,0){\line(2,1){20}}
  \put(30,0){\line(-3,1){30}}
}

\multiput(0,30)(0,10){1}{
  \put(0,0){\line(0,1){3}}
  \put(0,0){\line(1,1){3}}
  \put(10,0){\line(-1,1){3}}
  \put(10,0){\line(1,1){3}}
  \put(20,0){\line(-1,1){3}}
  \put(20,0){\line(1,1){3}}
  \put(30,0){\line(-1,1){3}}
  \put(30,0){\line(0,1){3}}

  \put(0,0){\line(2,1){6}}
  \put(10,0){\line(2,1){6}}
  \put(30,0){\line(-3,1){10}}
}

\multiput(0,0)(0,10){1}{
  \put(0,0){\line(0,-1){3}}
  \put(0,0){\line(1,-1){3}}
  \put(10,0){\line(-1,-1){3}}
  \put(10,0){\line(1,-1){3}}
  \put(20,0){\line(-1,-1){3}}
  \put(20,0){\line(1,-1){3}}
  \put(30,0){\line(-1,-1){3}}
  \put(30,0){\line(0,-1){3}}

  \put(0,0){\line(3,-1){10}}
  \put(20,0){\line(-2,-1){6}}
  \put(30,0){\line(-2,-1){6}}
}
\end{picture}
\end{center}
\caption{An adjacency matrix and its associated level poset.
Note that this is an example of a level Eulerian poset.}
\label{figure_main}
\end{figure}
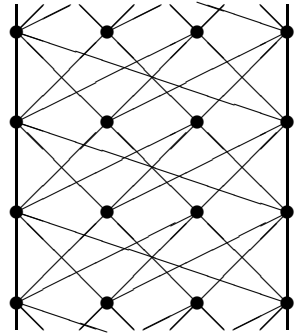

Informally speaking, the Hasse diagram of a level poset
can be thought of 
as a graph containing
a copy of the same vertex set $V$ at each ``level'' such that the portion
of the Hasse diagram containing the edges between elements of rank $i$
and rank $i+1$ may be obtained by vertically shifting the edges in the
Hasse diagram between the elements of rank $0$ and $1$. An example of a
level poset is shown in Figure~\ref{figure_main}. (The meaning of the term
Eulerian in this context will be explained in Section~\ref{sec:Eulerian}.)

Clearly it is sufficient to know the cover relations of the form
$(u,0)\prec (v,1)$ to obtain a complete description of a level
poset. Introducing the digraph $G$ with vertex set $V$ and edge set
$E:=\{(u,v)\in V\times V\::\: (u,0)<(v,1)\}$,  we obtain a digraph
representing a relation $E\subseteq V\times V$ on the vertex set $V$, i.e., a
digraph with no multiple edges but possibly containing loops. We call
$G$ the {\em underlying graph} of the level poset $P$ and $P$ the {\em
  level poset of $G$}. The poset $P$ and the digraph $G$
determine each other uniquely.
\begin{lemma}
\label{lemma_walks}
Let $P$ be a level poset on $V\times {\mathbb Z}$ and let $G$ be its underlying
digraph. Then for all $i,j\in{\mathbb Z}$ and for all $u,v\in V$ we have
$(u,i)<(v,j)$ in $P$ if and only if $i<j$ and there is a walk
$u=u_{0} \rightarrow u_{1} \rightarrowdots u_{j-i}=v$
of length $j-i$ from $u$ to $v$ in $G$.
\end{lemma}
The straightforward verification is left to the reader.
\begin{remark}
{\rm
By directing all the edges upwards in the Hasse diagram of $P$,
we obtain the {\em (right) derived graph} of the voltage
graph obtained from $G$ by assigning the voltage $1\in {\mathbb Z}$ to
each directed edge. For a
detailed discussion of the theory of voltage graphs, we refer the reader
to the work of Gross and Tucker~\cite{Gross_Tucker}. The
classical theory of voltage graphs focuses on the case where the
voltages belong to a finite group. Here we have to consider
${\mathbb Z}$, that is,
the simplest possible infinite group.
}
\end{remark}

Since the underlying digraph of a level poset $G$ is
uniquely determined by its adjacency matrix, every level poset is
uniquely determined by the adjacency matrix of its underlying
digraph. For brevity, we will use the term {\em underlying
matrix $M$} for ``adjacency matrix of the underlying digraph'' of a level
poset $P$, and the term {\em level poset of $M$} for
``level poset of the digraph whose adjacency matrix is~$M$''. The order
of the rows and columns of the
adjacency matrix corresponds to the 
order of vertices at the same level read from the left to the
right in a Hasse diagram of the corresponding level
poset. For any square matrix $M$ whose rows and columns are
indexed with elements of a set $V$, we will use the notation $M_{u,v}$
for the entry in the row indexed by $u\in V$ and in the column indexed
by $v\in V$.

Using Lemma~\ref{lemma_walks} we may describe the partial order of $P$ in
terms of its underlying matrix $M$ as follows.
\begin{corollary}
\label{corollary_walks}
Given a level poset $P$ with underlying matrix $M$, we have 
$(u,i)<(v,j)$ in $P$ if and only if $i<j$ and $M^{j-i}_{u,v}>0$ hold.
\end{corollary}

Using the operation $\bin$ we may rephrase Corollary~\ref{corollary_walks} as
follows.
\begin{corollary}
Given a level poset $P$ with underlying matrix $M$, we have 
$(u,i)<(v,j)$ in $P$ if and only if $i<j$ and $\bin(M^{j-i})_{u,v}=1$ hold.
\end{corollary}

Clearly the underlying digraph of a level poset is strongly connected if
and only if the underlying matrix $M$ is indecomposable. Equivalently,
for any pair of vertices $u,v\in V$ such that $u\neq v$ there is a $p>0$ such
that the adjacency matrix $M$ satisfies $\bin(M^{p})_{u,v}=1$.
If $|V|>1$ then the adjacency matrix $M$ of a
strongly connected digraph must also satisfy $\bin(M^{p})_{u,u}=1$ for some
$p$, given any $u\in V$. 

Having a strongly connected underlying digraph is
neither a necessary nor sufficient condition for the Hasse diagram of a
level poset (considered as an undirected graph) to be a connected
graph. An example of a connected level poset whose underlying digraph is
not strongly connected is the level poset with the underlying adjacency
matrix
$M=
\begin{pmatrix}
0 & 1 \\ 0 & 1
\end{pmatrix}$.
For level posets with strongly connected underlying digraphs, a 
necessary and sufficient condition for the connectivity of their Hasse
diagram may be stated using the notion of aperiodic
graphs.
A digraph on $n$ vertices is {\em aperiodic} if and only if
the underlying adjacency matrix $M$ is primitive.
\begin{theorem}
Assume that $P$ is the level poset of a strongly connected digraph
$G$. Then the Hasse diagram of $P$ is connected if and only
$G$ is aperiodic. 
\end{theorem}
\begin{proof}
Assume first $G$ is aperiodic and that its adjacency matrix $M$
satisfies $\bin(M^{p})=J$. Clearly $\bin(M^{n})=J$ for all
$n\geq p$. Given any $(u,i)$ and $(v,j)$ in $P$ there is a directed walk of
length $p+\max(i,j)-i$ from $u$ to $u$ and a directed walk of
length $p+\max(i,j)-j$ from $v$ to $u$ in~$G$. The first walk lifts to a
walk from $(u,i)$ to $(u,\max(i,j)+p)$ in the Hasse diagram of $P$,
whereas the second walk lifts to a walk from $(v,j)$ to
$(u,\max(i,j)+p)$ in the Hasse diagram. Thus we may walk from $(u,i)$ to
$(v,j)$ by first walking along the edges of the walk from $(u,i)$ to
$(u,\max(i,j)+p)$ and then following the edges of the walk from $(v,j)$ to
$(u,\max(i,j)+p)$ backwards.

Assume next that $G$ is not aperiodic. Let $k>1$ be an integer dividing
the length of every cycle. It is easy to see that
we may color the vertex set $V$ of $G$ using $k$ colors in such a way
that $(u,v)$ is an edge only if $u$ and $v$ have different colors. This
coloring may be lifted to the Hasse diagram of $P$ by setting the color
of $(u,i)$ to be the color of $u$ for each $(u,i)\in P$. There is no
walk between elements of the same color in the Hasse diagram of $P$.
\end{proof}

Let $\alpha = (\alpha_{1}, \ldots, \alpha_{r})$ be a composition
of~$m$, that is,
$\alpha_{1}, \ldots, \alpha_{r}$ are positive integers
whose sum is~$m$.
Let $S$ be the associated subset
of $\{1, \ldots, m-1\}$, that is,
$$ S = \{\alpha_{1}, \alpha_{1}+\alpha_{2},
             \ldots,
         \alpha_{1} + \cdots + \alpha_{r-1}\}  . $$
The flag $f$-vector entry $S$ of any interval
$[(u,i),(v,i+m)] \cong [(u,0),(v,m)]$ in
a level poset may computed using its underlying adjacency matrix as
follows.
\begin{lemma}
\label{lemma_flag_f}
Let $P$ be a level poset whose underlying digraph has vertex set $V$
of cardinality $n$.
Let $F_{S}$ be the $n \times n$ matrix whose $(u,v)$ entry is
$f_{S}([(u,0),(v,m)])$ if $(u,0) \leq (v,m)$ in $P$
and $0$ otherwise. Then the matrix~$F_{S}$
is given by
$$   F_{S} = \bin(M^{\alpha_{1}}) \cdot
             \bin(M^{\alpha_{2}}) \cdots
             \bin(M^{\alpha_{r}})              , $$
where $(\alpha_{1}, \ldots, \alpha_{r})$
is the composition associated with the subset
$S \subseteq \{1, \ldots, m-1\}$.
\end{lemma}
Note that every interval $[(u,0),(v,m)]$ in $P$ is isomorphic to all
intervals of the form $[(u,i),(v,i+m)]$ where $i\in {\mathbb Z}$
is an arbitrary integer.

\section{Level Eulerian posets}
\label{sec:Eulerian}

\begin{definition}
\label{definition_level_Eulerian}
We call a level poset $P$ a {\em level Eulerian poset} if every interval
is Eulerian.
\end{definition}
As a consequence of Lemma~\ref{lemma_flag_f}
we have the following condition for Eulerianness.
\begin{lemma}
A level poset is Eulerian
if and only if its adjacency matrix $M$
satisfies
\begin{equation}
   \sum_{i=0}^{p} (-1)^{i} \cdot \bin(M^{i}) \cdot \bin(M^{p-i}) = 0 
\label{equation_level_Eulerian}
\end{equation}
holds for all $p \geq 1$.
\label{lemma_level_Eulerian}
\end{lemma}
As it was noted in~\cite[Lemma~4.4]{Ehrenborg_k-Eulerian}
and~\cite[Lemma~2.6]{Ehrenborg_Readdy_b},
a graded poset of
odd rank is Eulerian if all of its proper intervals
are Eulerian. Thus it suffices to verify the condition
in Lemma~\ref{lemma_level_Eulerian} for even integers $p$. As a consequence
of Theorem~\ref{theorem_decomposable-block-matrix}, equation
\eqref{equation_level_Eulerian} only needs to be verified for
finitely-many values of $p$.
\begin{theorem}
Let $P$ be the level poset of an $n\times n$
indecomposable matrix $M$ with period $d$ and
index~$\gamma$.
Then $P$ is level Eulerian if and only if $M$ satisfies
the Eulerian condition~\eqref{equation_level_Eulerian}
for $p<2\gamma+4d$. For odd $d$, the bound for $p$ may be improved to 
$p<2\gamma+2d$. 
\label{theorem_Euler_bound}
\end{theorem}
\begin{proof}
We introduce $\Sigma(p)$ as a shorthand for $\sum_{i=0}^{p} (-1)^{i}
\cdot \bin(M^{i}) \cdot \bin(M^{p-i})$.
We wish to calculate 
$\Sigma(p+2d)-\Sigma(p)$ for an arbitrary $p\geq 2\gamma$. 

For $i=0,1,\ldots,\gamma-1$, the term $(-1)^{i} \cdot \bin(M^{i})
\cdot \bin(M^{p-i})$ in $\Sigma (p)$ cancels with the term $(-1)^{i} \cdot
\bin(M^{i}) \cdot \bin(M^{p+2d-i})$ in $\Sigma(p+2d)$
since $i<\gamma (d)$ and $p\geq
2\gamma$ imply $p-i\geq \gamma$. 
For $i=\gamma, \gamma+1,\ldots, p$, the term $(-1)^{i} \cdot
\bin(M^{i}) \cdot \bin(M^{p-i})$ in $\Sigma (p)$
cancels with the term
$(-1)^{i} \cdot \bin(M^{i+2d}) \cdot \bin(M^{p-i})$ in $\Sigma (p+2d)$
since $\bin(M^{i})=\bin(M^{i+2d})$ for $i\geq
\gamma$.   
After these cancellations, we obtain
\begin{equation}
\label{equation_p2d}
\Sigma(p+2d)-\Sigma(p)=\sum_{i=\gamma}^{\gamma+2d-1} (-1)^{i}
\cdot \bin(M^{i}) \cdot \bin(M^{p+2d-i})\quad\mbox{for $p\geq 2\gamma$}. 
\end{equation} 
If $d$ is odd, then the right-hand side of~\eqref{equation_p2d} is
zero. Indeed, for each $i$ satisfying $\gamma\leq i\leq \gamma+d-1$, 
the term $(-1)^{i} \cdot \bin(M^{i}) \cdot \bin(M^{p+2d-i})$ cancels with 
$(-1)^{d+i} \cdot \bin(M^{d+i}) \cdot \bin(M^{p+d-i})$ since $(-1)^d=-1$, 
$\bin(M^{i})=\bin(M^{d+i})$ and $\bin(M^{p+2d-i})=\bin(M^{p+d-i})$.
This concludes the proof of the theorem in the case when $d$ is odd. 

Assume from now on that $d$ is even. Substituting any $p\geq 2\gamma+2d$
in~\eqref{equation_p2d} yields 
\begin{align*}
\Sigma(p+2d)-\Sigma(p)
&=\sum_{i=\gamma}^{\gamma+2d-1} (-1)^{i} \cdot \bin(M^{i}) \cdot
\bin(M^{p+2d-i})\\ 
&=\sum_{i=\gamma}^{\gamma+2d-1} (-1)^{i} \cdot \bin(M^{i}) \cdot \bin(M^{p-i})
=\Sigma(p)-\Sigma(p-2d).\\ 
\end{align*}
We obtain that 
\begin{equation}
\label{equation_sigma_prec}
\Sigma(p+2d)-\Sigma(p)=\Sigma(p)-\Sigma(p-2d) \quad\mbox{holds for
  $p\geq 2\gamma+2d$.} 
\end{equation}
Therefore, if we verify that $\Sigma(p)=0$ holds for $p\leq
2\gamma+4d$, the equality $\Sigma(p)=0$ for $p>2(\gamma+d)$ may
be shown by induction on $p$ using~\eqref{equation_sigma_prec}. 
\end{proof}
As a consequence of Theorems~\ref{theorem_exp-primitive} and~\ref{theorem_Euler_bound}
we obtain the following upper bound.
\begin{corollary}
To determine whether the level poset $P$ of an $n\times n$ 
primitive
binary matrix $M$ is a level Eulerian poset one must only verify
the Eulerian condition~\eqref{equation_level_Eulerian} 
for $p\leq 2n^2-4n+6$.
\end{corollary}
\begin{remark}
{\rm
For a general $n \times n$
adjacency matrix it seems hard to give a better than
exponential estimate as a function of $n$ for the bounds
given in Theorem~\ref{theorem_Euler_bound}. However, for indecomposable
matrices we may still obtain a polynomial estimate using
Theorem~\ref{theorem_exp-indecomposable}. 
}
\end{remark}

\begin{example}
{\rm
The simplest level Eulerian poset is the butterfly poset
whose underlying adjacency matrix is
$$ M = \begin{pmatrix} 1 & 1 \\ 1 & 1 \end{pmatrix} . $$
This matrix has exponent $\gamma = 1$ and hence by
Theorem~\ref{theorem_Euler_bound} it is enough
to verify the Eulerian
condition~\eqref{equation_level_Eulerian} for $p=2$.
}
\label{example_order_2}
\end{example}

\begin{example}
{\rm
Consider the level poset shown in Figure~\ref{figure_main}. Its underlying
adjacency matrix $M$ satisfies 
$$
\bin(M^{2})=
\begin{pmatrix}
1&1&1&1\\
1&1&1&1\\
1&0&1&1\\
1&1&1&1\\
\end{pmatrix} $$
and $\bin(M^{3})=J$.
Thus $M$ is primitive and the exponent is given by
$\gamma=3$. 
Theorem~\ref{theorem_Euler_bound} gives the bound $p < 8$.
Hence to show that this matrix
produces a level Eulerian poset, we need verify
the Eulerian condition~\eqref{equation_level_Eulerian}
for the three values $p = 2, 4, 6$, which is a straightforward task.
}
\label{example_order_4}
\end{example}

Starting from the butterfly poset, for each $n\geq 2$ we may construct
a level Eulerian poset whose underlying digraph has $n$ vertices by
repeatedly using the following lemma.
The drawback to this construction is that it does
not add any more strongly connected components to the underlying graph.
\begin{lemma}
Let $M$ be a $n \times n$ matrix whose poset is level Eulerian
and let $\vec{v}$ be a column vector of~$M$.
Then\\[2mm] 
(i)
the level poset of the transpose matrix $M^{T}$ is
also level Eulerian.
\\
(ii)
the $(n+1) \times (n+1)$ matrix
$$  \begin{pmatrix} M & \vec{v} \\ 0 & 0 \end{pmatrix} $$ 
is also level Eulerian.
\end{lemma}

If we restrict our attention to level posets with strongly
connected underlying digraphs, we obtain the following restriction
on the order.

\begin{theorem}
Let $P$ be a level Eulerian poset whose underlying matrix $M$
is indecomposable.
Then the order of the matrix $M$ is even.
\label{theorem_even_order}
\end{theorem}
\begin{proof}
Let $d$ and $\gamma$ be respectively the period and index
of the matrix $M$.
Let $\delta$ be the least multiple of $d$ which is greater than or equal
to $\gamma$. (An upper bound for $\delta$ is $\gamma + d-1$.)
By reordering the vertices of the graph $G$,
we may assume that the matrix $M$ has the block form
given in equation~\eqref{equation_block-matrix}. 
Hence $M^{\delta}$ is also a block matrix.
Since $\delta \geq \gamma$, each block in 
$\bin(M^{\delta})$ is either the zero matrix or the matrix $J$
of all ones.
Since $\delta$ is a multiple of $d$,
the matrix $\bin(M^{\delta})$ has the form
\begin{equation}
\label{equation_M_to_delta}
\bin(M^{\delta})
   =
\begin{pmatrix}
J&0&\cdots&0\\
0&J&\cdots&0\\
\vdots&\vdots&\ddots&\vdots\\
0&0&\cdots&J\\
\end{pmatrix} ,
\end{equation}
where the $q$th block is a $c_{q}\times c_{q}$ square matrix
whose entries are all $1$'s. 
Apply the trace to 
the Eulerian condition~\eqref{equation_level_Eulerian}
for $p=2\delta$
and consider this equation modulo~$2$.
Recall $\trace(A B) = \trace(B A)$
holds for any pair of square matrices,
and in particular, it holds for
$A = \bin(M^{i})$ and
$B = \bin(M^{2\delta-i})$.
Hence the Eulerian condition~\eqref{equation_level_Eulerian}
collapses to
$$  \trace\left({\bin(M^{\delta})}^{2}\right) \equiv 0 \bmod 2  . $$
Note that the square of the
matrix $\bin(M^{\delta})$ is given by
$$
{\bin(M^{\delta})}^{2}
   =
\begin{pmatrix}
c_{0} \cdot J&0&\cdots&0\\
0&c_{1} \cdot J&\cdots&0\\
\vdots&\vdots&\ddots&\vdots\\
0&0&\cdots&c_{q-1} \cdot J\\
\end{pmatrix} . $$
Hence the trace of the above matrix is
$\sum_{q=0}^{d-1} c_{q}^{2} \equiv \sum_{q=0}^{d-1} c_{q} \bmod 2$.
Hence we conclude that the order of $M$ is an even number.
\end{proof}

\section{Level half-Eulerian posets}

In analogy to level Eulerian posets
(Definition~\ref{definition_level_Eulerian}) we define level
half-Eulerian posets as follows. 
\begin{definition}
\label{definition_level_half-Eulerian}
A level poset $P$ is said to be a
{\em level half-Eulerian poset} if every interval
is half-Eulerian.
\end{definition}
In analogy to the horizontal doubling operation introduced
in~\cite{Bayer_Hetyei_E,Bayer_Hetyei_G}, we define the {\em horizontal
double $D_{\leftrightarrow}(P)$ of a level poset $P$} as the poset
obtained by replacing each $(u,i)\in P$ by two copies $(u_{1},i)$, $(u_{2},i)$
and preserving the partial order of $P$, i.e., setting $(u_{k},i)<(v_{l},j)$
in $D_{\leftrightarrow}(P)$ if and only if $(u,i)<(v,j)$ holds in $P$. 
\begin{proposition}
The horizontal double $D_{\leftrightarrow}(P)$ of a level poset $P$ is a
level poset. In particular, if the underlying adjacency matrix of $P$ is
$M$ then the underlying adjacency matrix $D_{\leftrightarrow}(M)$ of
$D_{\leftrightarrow}(P)$ is
\begin{equation}
\label{equation_horizontal_doubling_matrix}
D_{\leftrightarrow}(M)=
\begin{pmatrix}
M & M \\ M & M 
\end{pmatrix} .
\end{equation}
\end{proposition}
The straightforward verification is left to the reader. As an immediate
consequence of Lemma~\ref{lemma_h_double}, we obtain the following corollary.
\begin{corollary}
\label{corollary_h_double}
A level poset $P$ is half-Eulerian if and only if its horizontal
double $D_{\leftrightarrow}(P)$ is level Eulerian. 
\end{corollary}
Lemma~\ref{lemma_level_Eulerian} has the following half-Eulerian
analogue.
\begin{lemma}
A level poset is half-Eulerian
if and only if its adjacency matrix $M$
satisfies
\begin{equation}
   \sum_{i=1}^{p-1} (-1)^{i-1} \cdot \bin(M^{i}) \cdot \bin(M^{p-i}) = 
\left\{
\begin{array}{ll}
J&\mbox{if $p$ is even,}\\
0&\mbox{if $p$ is odd,}\\
\end{array}
\right. 
\label{equation_level_half-Eulerian}
\end{equation}
for all $p \geq 1$.
\label{lemma_level_half-Eulerian}
\end{lemma}
Directly from
Lemmas~\ref{lemma_h_double} and~\ref{lemma_flag_f},
a level poset $P$ is half-Eulerian if and only if its underlying
adjacency matrix $M$ satisfies~\eqref{equation_level_half-Eulerian} for
all $p>0$. It is straightforward to show 
directly that the adjacency matrix
$M$ of a level poset $P$ satisfies~\eqref{equation_level_half-Eulerian}
for a given $p>0$ if and only if the matrix $D_{\leftrightarrow}(M)$
given in~\eqref{equation_horizontal_doubling_matrix}
satisfies~\eqref{equation_level_Eulerian} for
the same $p$. As a consequence, we only need to verify
\eqref{equation_level_half-Eulerian} for values of $p$ up to the bound stated 
in Theorem~\ref{theorem_Euler_bound}. 

Unfortunately 
the natural generalization of the vertical doubling operation to a level
poset by replacing each element $(u,i)$ of a level poset $P$
with two copies $(u_{1},2i)$ and $(u_{2},2i+1)$, 
setting $(u_{1},2i)<(u_{2},2i+1)$ for each $u$,
and setting $(u_{2},2i+1)<(v_{1},2j)$ whenever
$(u,i)<(v,j)$ does not work.
This operation does not result in
a level poset because
the cover relations between levels
$2i$ and $2i+1$ would be different from the cover relations between
levels $2i-1$ and $2i$. 
However, we may perform this operation, take two
copies, shift the Hasse diagram of one of the copies
one step up, and finally intertwine the two copies.
In short, consider the level poset with the adjacency matrix
\begin{equation}
\label{equation_vertical_doubling_matrix}
D_{\updownarrow}(M) = \begin{pmatrix} 0&I\\ M&0 \end{pmatrix} ,
\end{equation}
where $M$ is the $n\times n$ underlying adjacency matrix
of the level poset $P$ and $I$ is the $n\times n$ identity matrix. 
\begin{definition}
Let $P$ be a level poset with adjacency matrix $M$. 
Define the {\em vertical double $D_{\updownarrow}(P)$} 
of $P$ to be the level poset whose
adjacency matrix $D_{\updownarrow}(M)$ is given
by~\eqref{equation_vertical_doubling_matrix}. 
\end{definition}
As a direct consequence of Lemma~\ref{lemma_v_double}, we have the following
corollary. 
\begin{corollary}
\label{corollary_v_double}
The vertical double $D_{\updownarrow}(P)$ of an arbitrary level poset
$P$ is level half-Eulerian.
\end{corollary}
Combining Corollaries~\ref{corollary_h_double}
and~\ref{corollary_v_double}, we obtain the
following statement.
\begin{corollary}
\label{corollary_h_v_double}
For any square binary matrix $M$, the matrix 
$$
D_{\leftrightarrow}(D_{\updownarrow} (M))=
\begin{pmatrix}
0&I&0&I\\
M&0&M&0\\
0&I&0&I\\
M&0&M&0
\end{pmatrix}
$$
is the adjacency matrix of a level Eulerian poset.
\end{corollary}
\begin{remark}
{\rm
The vertical doubling operation induces a widely used
operation on the underlying digraph. If $G$ is a digraph with adjacency
matrix $M$ 
then $D_{\updownarrow}(M)$ is the adjacency matrix of the digraph
$D_{\updownarrow}(G)$ obtained from $G$ as follows.
\begin{enumerate}
\item Replace each vertex $u$ of $G$ with two copies $u_{1}$ and $u_{2}$. 
\item The edge set of $D_{\updownarrow}(G)$ consists of all edges of the
  form $u_{1}\rightarrow u_{2}$ and of all edges of the form $u_{2}\rightarrow
  v_{1}$ where $u\rightarrow v$ is an edge in $G$.  
\end{enumerate}
Introducing a graph identical to or very similar to $D_{\updownarrow}(G)$ 
is often used in the study of network flows.
These type of constructions  also appear in proofs of the
vertex-disjoint path variant of Menger's theorem
as a way to reduce the
study of vertex capacities to that of edge capacities.
}
\end{remark}

Every half-Eulerian poset arising as a vertical double of a level poset
has an even number of elements at each level, and each canonical block of its
underlying adjacency matrix has even period. This observation may be
complemented by the following analogue of
Theorem~\ref{theorem_even_order}
for half-Eulerian posets.
\begin{theorem}
\label{theorem_odd_order}
Let $P$ be a level half-Eulerian poset
whose underlying matrix $M$ is primitive.
Then the order of the matrix $M$ is odd. 
\end{theorem}
\begin{proof}
Since $M$ is primitive,
let $\gamma$ be the exponent of the matrix $M$.
Recall that $\bin(M^{\gamma})=\bin(M^{\gamma+1})=\cdots=\bin(M^{2\gamma})=J$.
Hence we may rewrite the half-Eulerian
condition~\eqref{equation_level_half-Eulerian}
for $p=2\gamma$
as
\begin{equation}
\label{equation_p-even}
X J + (-1)^{\gamma-1} J^{2} + J X = J,
\end{equation}
where 
$X=\sum_{i=1}^{\gamma-1} (-1)^{i-1} \cdot \bin(M^{i})$.
Similarly, the half-Eulerian condition~\eqref{equation_level_half-Eulerian}
for $p=2\gamma+1$ yields
\begin{equation}
\label{equation_p-odd}
X J - J X = 0.
\end{equation}
Combining equations~\eqref{equation_p-even} and~\eqref{equation_p-odd}
modulo $2$
yields that
$J^{2} \equiv J \bmod 2$.
Since $J^{2} = n J$,
the order $n$ must be odd.
\end{proof}

It is interesting to note that
unlike the proof of Theorem~\ref{theorem_even_order},
the trace operation does not appear
in the argument for
Theorem~\ref{theorem_odd_order}.

\section{Shellable level posets}

Labelings that induce a shelling of the order complex
of a graded poset
have a vast literature. In the case of level posets it is natural to
seek a labeling that may be defined in a uniform fashion for the order
complex of every interval. The next definition is an example of such a
uniform labeling.
\begin{definition}
\label{definition_vertex_shelling}
Let $G$ be a directed graph on the vertex set $V$. A linear
order on $V$ is a {\em vertex shelling order} if for any $u,v\in V$ and
every pair of walks
$u=v_{0}\rightarrow v_{1} \rightarrowdots v_{k}=v$
and $u=v_{0}'\rightarrow v_{1}'\rightarrowdots v_{k}'=v$
of the same length such that $v_{1}'<v_{1}$ holds, there is a $j\in [1,k-1]$
and a vertex $w\in V$ such that $w<v_{j}$ holds and $v_{j-1}\rightarrow w$ and
$w\rightarrow v_{j+1}$ are edges of $G$. 
\end{definition} 
The term ``vertex shelling order'' is justified by the following
result. 
\begin{theorem}
\label{theorem_vertex_shelling}
Let $P$ be a level poset and $<$ be a vertex shelling order on the
vertex set of the underlying digraph of $P$.
Associate to each
maximal chain $(u,i)=(v_{i},i)\prec (v_{i+1},i+1)\prec \cdots \prec
(v_{j-1},j-1)\prec (v_{j},j)=(v,j)$ in the interval $[(u,i),(v,j)]$
the word $v_{i} \cdots v_{j}$.
Then ordering the maximal chains of the interval $[(u,i),(v,j)]$
in $P$ by increasing lexicographic order of
the associated words is a shelling
of the order complex $\Delta([(u,i),(v,j)])$.
\end{theorem}
\begin{proof}
The maximal chains of an interval $[(u,i),(v,j)]$ in a level poset $P$
are in a one-to-one correspondence with walks $u=v_{i}\rightarrow
v_{i+1}\rightarrowdots v_{j-1}\rightarrow v_{j}=v$
of length $j-i$ in the underlying digraph. Assume that the maximal chain encoded
by the word $v_{i}\cdots v_{j}$ is preceded by the chain encoded by
the word $v_{i}^{\prime} \cdots v_{j}^{\prime}$.
Let $k \in [i+1,j-1]$ be the least
index such that $v_{k}^{\prime} \neq v_{k}$ and let $l \in [k+1,j]$
be the least index such
that $v_{l} = v_{l}^{\prime}$.
Since $v_{i} \cdots v_{j}$ is preceded by $v_{i}^{\prime} \cdots
v_{j}^{\prime}$ in the lexicographic order,
we must have $v_{k}^{\prime} < v_{k}$.
As a consequence of Definition~\ref{definition_vertex_shelling} applied to the walks 
$v_{k-1}\rightarrowdots v_{l}$ and
$v_{k-1}^{\prime} \rightarrow \cdots \rightarrow v_{l}^{\prime}$,
there is an $m\in [k,l-1]$ and a vertex $w$
such that $w<v_m$ holds and $v_{m-1}\rightarrow w$ and
$w \rightarrow v_{m+1}$ are edges of the underlying digraph.
The maximal chain
associated to the word $v_{i}v_{i+1}\cdots v_{m-1} w v_{m+1}\cdots v_{j}$
precedes the maximal chain associated to $v_{i}\cdots v_{j}$, the
intersection of the two chains has codimension one, and contains the 
intersection of the chain associated to $v_{i}\cdots v_{j}$ with the
chain associated to $v_{i}^{\prime} \cdots v_{j}^{\prime}$.   
\end{proof}
\begin{remark}
{\rm
The shelling order used in Theorem~\ref{theorem_vertex_shelling} is induced by
labeling each cover relation $(u,i)\prec(v,i+1)$ by the vertex $v$. This
is an example of Kozlov's {\em $CC$-labelings}~\cite{Kozlov},
discovered independently by Hersh and Kleinberg.
See the Introduction of~\cite{Babson_Hersh}.
Even if the linear order
$<$ is not a vertex shelling order, labeling each cover relation
$(u,i) \prec (v,i+1)$ by the vertex $v$ and
using the linear order $<$ to
lexicographically order the maximal chains in each interval induces an
{\em $FA$-labeling}, as defined by Billera and
Hetyei~\cite{Billera_Hetyei_planar}.
See also \cite{Billera_Hetyei_flag}.
Essentially the same labelings were
used by Babson and Hersh~\cite{Babson_Hersh} to construct a discrete
Morse matching, a technique
which helps determine the homotopy type of the order
complex of an arbitrary graded poset. We refer the reader to the above
cited sources for further information.
}
\end{remark} 
In analogy to Theorem~\ref{theorem_Euler_bound} the
vertex shelling order condition 
needs to be verified only for finitely-many
values of $k$. We prove this for strongly connected digraphs. 
\begin{theorem}
\label{theorem_vertex_shelling_bound}
Let $P$ be the level poset of an indecomposable $n\times n$ matrix $M$
with period~$d$ and index~$\gamma$.
Then the underlying digraph~$G$ of $P$ is vertex 
shellable if and only if it satisfies the vertex shelling order
condition stated in
Definition~\ref{definition_vertex_shelling} for $k\leq \gamma+d$.
\end{theorem}
\begin{proof}
Assume $G$ satisfies the condition stated in Definition~\ref{definition_vertex_shelling}
for $k\leq \gamma+d$ and let $k$ be the least integer for which the
condition is violated. We must
have $k\geq \gamma+d+1$. Consider any pair of walks
$u=v_{0}\rightarrow v_{1}\rightarrowdots v_{k}=v$ and
$u=v_{0}'\rightarrow v_{1}'\rightarrowdots v_{k}'=v$ such that
$v_{1}'<v_{1}$.  As in the proof of Theorem~\ref{theorem_even_order}, let
$\delta\leq \gamma + d-1$ be the least multiple of $d$ which is greater
than or equal to $\gamma$. After rearranging the rows and columns if
necessary, $\bin(M)$ takes the form given in~\eqref{equation_block-matrix} and 
$\bin(M^{\delta})$ takes the form given in~\eqref{equation_M_to_delta}.
As in~\eqref{equation_block-matrix}, we may assume that the block
$Q_{q,q+1}$ occupies the rows indexed by $C_{q}$ and the columns indexed
by $C_{q+1}$. The vertex set of $G$ is the
disjoint union of the sets $C_{0},\ldots, C_{d-1}$ and every edge starting 
in $C_q$ ends in~$C_{q+1}$. Since there is a walk of length $k-1$ from 
$v_{1}'$ to $v_{k}=v$ and there is a walk of length $k-\delta-1$ from
$v_{\delta+1}$
to $v_{k}$, it follows that $v_{1}'$ and $v_{\delta+1}$ belong to the same
set $C_q$. As a consequence of equation~\eqref{equation_M_to_delta}, there is a
walk $v_{1}'\rightarrow v_{2}''\rightarrowdots
v_{\delta}''\rightarrow v_{\delta+1}$ of length $\delta$ from $v_{1}'$ to
$v_{\delta+1}$. Note that $\delta+1\leq \gamma+d<k$. By the minimality
of $k$, the walks $v_{0}\rightarrow v_{1}\rightarrowdots
v_{\delta}\rightarrow v_{\delta+1}$ and 
$v_{0}\rightarrow v_{1}'\rightarrow v_{2}''\rightarrow
\cdots \rightarrow v_{\delta}''\rightarrow v_{\delta+1}$ still
satisfying $v_{1}'<v_{1}$ cannot violate the 
vertex shelling order
condition stated in
Definition~\ref{definition_vertex_shelling}. Hence there is a $j\in [1,\delta]$
and a vertex $w\in V$ such that $w<v_{j}$ holds and $v_{j-1}\rightarrow w$ and
$w\rightarrow v_{j+1}$ are edges of $G$, and the pair of walks
$u=v_{0}\rightarrow v_{1}\rightarrowdots v_{k}=v$ and 
$u=v_{0}'\rightarrow v_{1}'\rightarrowdots v_{k}'=v$ do not
violate the vertex shelling order condition,
in contradiction with our assumption. 
\end{proof}
The verification whether a linear order on the vertices of a digraph is
a vertex shelling order may be automated by introducing the algebra
  of walks.

\begin{definition}
Let $G$ be a digraph with edge set $E$ on the vertex set $V$.
Assume $G$ has no multiple edges. The {\em algebra of walks
$\walks{\Qqq}{G}$} is the quotient of the free non-commutative algebra
over $\Qqq$ generated by the set of variables $\{x_{u,v}\::\: (u,v)\in E\}$
by the ideal generated by the set of monomials
$\{x_{u_{1},v_{1}}x_{u_{2},v_{2}}\::\:v_{1}\neq u_{2} \}$.   
\end{definition}
A vector space basis for $\walks{\Qqq}{G}$ may be given by $1$,
which labels the trivial walk, and all 
monomials $x_{v_{0},v_{1}}x_{v_{1},v_{2}}\cdots x_{v_{k-1},v_{k}}$
such that $v_{0}\rightarrow v_{1}\rightarrowdots v_{k}$ is a walk in $G$.  
\begin{notation}
We introduce $x_{v_{0},v_{1},\ldots, v_{k}}$ as a
shorthand for $x_{v_{0},v_{1}}x_{v_{1},v_{2}}\cdots x_{v_{k-1},v_{k}}$.
\end{notation}

\begin{theorem}
\label{theorem_walks}
Let $G$ be a digraph on the vertex set $V$ 
of cardinality $n$ having no multiple edges and
let $<$ be a linear order on $V$. Let $I_{<}$
be the ideal in $\walks{\Qqq}{G}$
generated by all monomials $x_{v_{0},v_{1}}x_{v_{1},v_{2}}$ such that
there is a vertex $v_{1}^{\prime} < v_{1}$
such that $v_{0} \rightarrow v_{1}^{\prime} \rightarrow v_{2}$
is a walk in $G$.
Let $Z=(z_{u,v})_{u,v\in V}$ be the
$n \times n$ matrix whose rows and columns are indexed by the vertices of
$G$ such that $z_{u,v}=x_{u,v}$ if $(u,v)$ is an edge and it is zero
otherwise. If over the
ring $\walks{\Qqq}{G}/I_{<}$ every entry in 
every power of the matrix $Z$ is a single monomial or zero,
then $<$ is a vertex shelling order.
\end{theorem}
\begin{proof}
We first calculate the powers of the matrix $Z$ over the ring
$\walks{\Qqq}{G}$. It is straightforward to see by induction on $k$
that the entry $z^{(k)}_{u,v}$ in $Z^k$ is the sum of all monomials 
of the form $x_{v_{0},v_{1},\ldots ,v_{k}}$ where 
$u=v_{0}\rightarrow v_{1}\rightarrow \cdots \rightarrow v_{k}=v$
is a walk of length $k$ from $u$ to $v$. 

The effect of factoring by the ideal $I_{<}$
may be easily described by introducing 
the following {\em flip operators} $\sigma_{i}$ for $i\geq 1$.
Given a monomial $x_{v_{0},v_{1},\ldots ,v_{k}}$, 
set
$$   \sigma_{i}(x_{v_{0},v_{1},\ldots ,v_{k}})
   =
     x_{v_{0},v_{1},\ldots,v_{i-1},v_{i}^{\prime},v_{i+1},\ldots,v_{k}} , $$
if $v_{i}^{\prime}$
is the least vertex in the linear order
such that $v_{i-1} \rightarrow v_{i}^{\prime} \rightarrow v_{i+1}$
is a walk in the digraph.
Clearly a monomial belongs to $I_{<}$ if and only if 
is not fixed by some $\sigma_{i}$. The order $<$ induces a lexicographic
order on all walks of length $k$ from $u$ to $v$. Applying a flip 
$\sigma_{i}$ to a monomial $x_{v_{0},v_{1},\ldots,v_{k}}$
either leaves the monomial unchanged or replaces it with a monomial
that represents a lexicographically smaller walk of length $k$ from $u$
to $v$. In particular, the monomial representing the lexicographically
least walk of length $k$ from $u$ to $v$ does not belong to $I_{<}$ and so 
$z^{(k)}_{u,v}\neq 0$ in $\walks{\Qqq}{G}/I_{<}$ if there is a walk 
of length $k$ from $u$ to $v$. 

Assume first that each $z^{(k)}_{u,v}$ contains at most one monomial that
does not belong to $I_{<}$. As noted above, in this case $z^{(k)}_{u,v}$
contains exactly one monomial not belonging to $I_{<}$ and this monomial
represents the lexicographically least walk of length $k$ from $u$ to
$v$. Consider a pair of walks
$u=v_{0}\rightarrow v_{1}\rightarrowdots v_{k}=v$ and
$u=v_{0}'\rightarrow v_{1}'\rightarrowdots v_{k}'=v$
satisfying $v_{1}'<v_{1}$. Since
$u=v_{0}\rightarrow v_{1}\rightarrowdots v_{k}=v$ is not the lexicographically
least walk of length $k$ from $u$ to $v$, the monomial 
$x_{v_{0},v_{1},\ldots,v_{k}}$ must belong to
$I_{<}$. Thus there is a $j\in [1,k-1]$ and a vertex $w<v_{j}$ such that
$$    \sigma_{j}(x_{v_{0},v_{1},\ldots,v_{k}})
    =
      x_{v_{0},v_{1},\ldots,v_{j-1},w,v_{j+1},\ldots,v_{k}} . $$
This $j$ and $w$ show that the
vertex shelling order
condition stated in
Definition~\ref{definition_vertex_shelling} is satisfied.   

Assume that $<$ is a vertex shelling order and, by way of
contradiction, assume that $z^{(k)}_{u,v}$ contains at least two monomials 
$x_{v_{0},v_{1},\ldots,v_{k}}$ and
$x_{v_{0},v_{1}',\ldots, v_{k-1}',v_{k}}$ not belonging to $I_{<}$.
Without loss of generality we may assume $v_{0}=v_{0}', v_{1}=v_{1}',
\ldots, v_{i-1}=v_{i-1}'$ and $v_{i}'<v_{i}$. Applying the
vertex shelling order condition given
to the pair of walks $v_{i-1}\rightarrow
v_{i}\rightarrowdots v_{k}$ and  $v_{i-1}\rightarrow
v_{i}'\rightarrowdots v_{k}$, we obtain a $j\in
[i,k-1]$ and a vertex $w$ such that $w<v_{j}$ and $v_{j-1}\rightarrow w$
and $w\rightarrow v_{j+1}$ are edges. But then 
$\sigma_{j}(x_{v_{0},v_{1},\ldots,v_{k}})\neq
x_{v_{0},v_{1},\ldots,v_{k}}
$
and $x_{v_{0},v_{1},\ldots,v_{k}}$ does not belong to $I_{<}$.
\end{proof}
Using Theorem~\ref{theorem_vertex_shelling_bound}, the proof of Theorem~\ref{theorem_walks}
may be modified to show the following.
\begin{proposition}
\label{proposition_walks_bound}
Let $M$ be an indecomposable matrix having
period $d$ and index $\gamma$. When applying
Theorem~\ref{theorem_walks} to decide whether an order of the vertices is a
vertex shelling order, one needs to verify the condition on the matrices
$Z^{k}$ only for $k \leq \gamma+d$.
\end{proposition}

For a shellable Eulerian posets we can conclude more.
\begin{theorem}
If $P$ is an Eulerian poset of rank $n+1$ whose order complex
$\Delta(P - \{\hz,\ho\})$ is shellable, then
the order complex is homeomorphic to an $n$-dimensional sphere.
\label{theorem_shellable_Eulerian}
\end{theorem}
\begin{proof}
Since every interval of rank $2$ in an Eulerian poset
is a diamond, every subfacet of the order complex
is contained in exactly two facets. Hence the
order complex $\Delta(P - \{\hz,\ho\})$ is a pseudo-manifold
without boundary.
Let $F_{1}, \ldots, F_{t}$ be a shelling of the order complex.
Since the reduced Euler characteristic equals
the M\"obius function $\mu(P)$, which in turn equals $(-1)^{n+1}$
as $P$ is Eulerian,
the order complex is homotopy equivalent to one sphere.
Hence there is one facet that changes the topology
during shelling.
We can move this facet to be last facet of the shelling,
that is, $F_{t}$.
Hence the previous facets form
a contractible complex.
The shelling implies that
the complex $F_{1} \cupdots F_{t-1}$ is collapsible to a point.
Lastly, the complex $F_{1} \cupdots F_{t-1}$ is
a pseudo-manifold with boundary.
By a result of J.H.C.\ Whitehead~\cite[Theorem~1.6]{Forman},
such a pseudo-manifold is homeomorphic to an $n$-dimensional ball.
The result follows by gluing back the last facet $F_{t}$ along
the common boundary.
\end{proof}

\begin{example}
{\rm
Consider again the level poset shown in Figure~\ref{figure_main}.
See also Example~\ref{example_order_4}.
Consider the linear
order $1<2<3<4$ where $i$ is the vertex associated to row (and column)
$i$. We claim this is a vertex shelling order. By 
Proposition~\ref{proposition_walks_bound} we need to check that all entries of
$Z^k$ are zero or monomials for $k\leq 4$.  
Direct calculation shows
$$
Z=
\begin{pmatrix}
x_{1,1} & x_{1,2} & x_{1,3} & 0\\
x_{2,1} & 0       & x_{2,3} & x_{2,4}\\
0       & x_{3,2} & 0       & x_{3,4}\\
x_{4,1} & 0       & x_{4,3} & x_{4,4}
\end{pmatrix}
,\quad
Z^{2}=
\begin{pmatrix}
x_{1,1,1} & x_{1,1,2} & x_{1,1,3} & x_{1,2,4}\\
x_{2,1,1} & x_{2,1,2} & x_{2,1,3} & x_{2,3,4}\\
x_{3,2,1} & 0         & x_{3,2,3} & x_{3,2,4}\\
x_{4,1,1} & x_{4,1,2} & x_{4,1,3} & x_{4,3,4}
\end{pmatrix}
$$
$$
Z^{3}=
\begin{pmatrix}
x_{1,1,1,1} & x_{1,1,1,2} & x_{1,1,1,3} & x_{1,1,2,4}\\
x_{2,1,1,1} & x_{2,1,1,2} & x_{2,1,1,3} & x_{2,1,2,4}\\
x_{3,2,1,1} & x_{3,2,1,2} & x_{3,2,1,3} & x_{3,2,3,4}\\
x_{4,1,1,1} & x_{4,1,1,2} & x_{4,1,1,3} & x_{4,1,2,4}
\end{pmatrix}
\quad\mbox{and}
$$
$$
Z^{4}=
\begin{pmatrix}
x_{1,1,1,1,1} & x_{1,1,1,1,2} & x_{1,1,1,1,3} & x_{1,1,1,2,4}\\
x_{2,1,1,1,1} & x_{2,1,1,1,2} & x_{2,1,1,1,3} & x_{2,1,1,2,4}\\
x_{3,2,1,1,1} & x_{3,2,1,1,2} & x_{3,2,1,1,3} & x_{3,2,1,2,4}\\
x_{4,1,1,1,1} & x_{4,1,1,1,2} & x_{4,1,1,1,3} & x_{4,1,1,2,4}
\end{pmatrix}  . $$
Thus we conclude that each interval is shellable
and hence by Theorem~\ref{theorem_shellable_Eulerian}
that the order complex of each interval is
homeomorphic to a sphere.
}
\label{example_order_4_shelling}
\end{example}

We conclude this section with an example of a level Eulerian poset that
has a strongly connected underlying digraph and non-shellable intervals. 

\begin{example}
{\rm
Consider the level poset $P$ whose underlying graph has the adjacency matrix 
$$
M=
\begin{pmatrix}
1 & 1 & 0\\
0 & 0 & 1\\
1 & 0 & 0\\
\end{pmatrix}
. $$
We then have
$$
\bin(M^{2})=
\begin{pmatrix}
1 & 1 & 1\\
1 & 0 & 0\\
1 & 1 & 0\\
\end{pmatrix} ,
\quad
\bin(M^{3})=
\begin{pmatrix}
1 & 1 & 1\\
1 & 1 & 0\\
1 & 1 & 1\\
\end{pmatrix}
$$
and
$\bin(M^{4})=J$.
Thus $M$ is primitive and the exponent is $\gamma=4$.
By the half-Eulerian
analogue of Theorem~\ref{theorem_Euler_bound}, to check whether $P$ is 
half-Eulerian we only need to verify the half-Eulerian
condition~\eqref{equation_level_half-Eulerian}
for $p<10$. Furthermore, just as
in the Eulerian case, we only need to
check~\eqref{equation_level_half-Eulerian}
holds for even values of $p$.  
We leave this to the reader as an exercise. The level poset $P$ is 
half-Eulerian, so its horizontal double $D_{\leftrightarrow}(P)$ is
Eulerian. Let $1$ denote the vertex corresponding
to the first row in $M$
and consider the interval $[(1,0),(1,3)]$ in
$D_{\leftrightarrow}(P)$. The order complex of this interval has two
connected components and is thus not shellable.  
}
\end{example}

\section{The $\ab$- and $\cd$-series of level and level Eulerian posets}

Level Eulerian posets are infinite in nature, so
one must encode their face incidence data using a non-commutative
series.
For a reference on non-commutative formal power series,
see~\cite[Section~6.5]{Stanley_EC_2}.
In this section we review the notions of the flag $h$-vector,
the $\ab$-index for finite posets
and,
in the case the poset is Eulerian,
the $\cd$-index.
We extend these notions to the $\ab$-series and
$\cd$-series of a level Eulerian poset.
The main result of this section is that
the $\cd$-series of any level
Eulerian poset is a rational generating function.

For a finite graded poset $P$ of rank $m+1$,
the flag $f$-vector has $2^{m}$
entries. When the poset $P$ is Eulerian, there are linear relations
among these entries known as the generalized
Dehn--Sommerville relations~\cite{Bayer_Billera}.
They describe a subspace whose dimension
is given by the $m$th Fibonacci number. The $\cd$-index offers
an explicit bases for this subspace. In order to describe it, we
begin by defining the flag $h$-vector and the $\ab$-index.
The {\em flag $h$-vector} of the poset $P$
is defined by the invertible relation
$$      h_{S} = \sum_{T \subseteq S} (-1)^{|S-T|} \cdot f_{T} . $$
Hence the flag $h$-vector encodes the same information as the
flag $f$-vector.
Let $\av$ and $\bv$ be two non-commutative variables
each of degree one.
For $S$ a subset of $\{1, \ldots, m\}$ define the
$\ab$-monomial $u_{S} = u_{1} u_{2} \cdots u_{m}$
by letting $u_{i} = \bv$ if $i \in S$ and
$u_{i} = \av$ otherwise. The {\em $\ab$-index} of the poset $P$
is defined by
$$    \Psi(P)
    =
      \sum_{S} h_{S} \cdot u_{S}   ,  $$
where the sum is over all subsets
$S \subseteq \{1, \ldots, m\}$.

Bayer and Klapper~\cite{Bayer_Klapper} proved that
for an Eulerian poset $P$ the $\ab$-index can be written
in terms of the non-commutative variables $\cv = \av + \bv$ and
$\dv = \av\bv + \bv\av$ of degree
one and two, respectively.
There are several proofs of this fact in the
literature~\cite{Ehrenborg_k-Eulerian,Ehrenborg_Readdy_homology,Stanley_d}.
When $\Psi(P)$ is written in terms of $\cv$ and $\dv$,
it is called the {\em $\cd$-index} of the poset $P$.

Another way to approach the $\ab$-index is by
chain enumeration. For a chain in the poset $P$
$c = \{\hz = x_{0} < x_{1} < \cdots < x_{k+1} = \ho\}$
define its weight to be
$$  \wt(c)
       =
    (\av-\bv)^{\rho(x_{0},x_{1}) - 1}
       \cdot
    \bv
       \cdot
    (\av-\bv)^{\rho(x_{1},x_{2}) - 1}
       \cdot
    \bv
       \cdots
    \bv
       \cdot
    (\av-\bv)^{\rho(x_{k},x_{k+1}) - 1}    .   $$
The $\ab$-index is then given by
\begin{equation}
     \Psi(P) = \sum_{c} \wt(c)    ,
\label{equation_chain_definition}
\end{equation}
where the sum ranges over all chains $c$ in the poset $P$.

For a level poset $P$ and two vertices $x$ and $y$ in the underlying
digraph, set $\Psi([(x,i),(y,j)])$ to be zero if
$(x,i) \not\leq (y,j)$. Define the {\em $\ab$-series
$\Psi_{x,y}$} of the level poset $P$
to be the non-commutative formal power series
$$   \Psi_{x,y}
   =
     \sum_{m \geq 0} \Psi([(x,0),(y,m+1)])   .   $$
Since the $m$th term in this sum is homogeneous of degree $m$,
the sum is well-defined.

Finally, for a level poset $P$ let $\Psi$ be the matrix
whose $(x,y)$ entry is
the $\ab$-series $\Psi_{x,y}$.
Our goal is to show 
that the $\ab$-series $\Psi_{x,y}$
is a rational non-commutative formal power series.

Let $K(t)$ denote the matrix
$$ K(t) = M + \bin(M^{2}) \cdot t + \bin(M^{3}) \cdot t^{2}
                        +  \cdots , $$
and let $K_{x,y}(t)$ denote the $(x,y)$ entry of the matrix $K(t)$.
Observe $K_{x,y}(t)$ is the generating function having the
coefficient of $t^{m-1}$ to be $1$ if there is a walk of
length $m$ from
the vertex $x$ to the vertex $y$ and zero otherwise.

To prove the main result of this section we need
the following classical result
due to
Skolem~\cite{Skolem},
Mahler~\cite{Mahler} and
Lech~\cite{Lech}.
The formulation here is
the same as that given
in~\cite[Chapter~4, Exercise~3]{Stanley_EC_1}).
In fact, since we are only dealing with integer coefficients,
it is sufficient to use
Skolem's original result~\cite{Skolem}.
\begin{theorem}[Skolem--Mahler--Lech]
Let $\sum_{n \geq 0} a_{n} \cdot t^{n}$ be a rational generating
function and let
$$    b_{n}
   =
      \left\{ \begin{array}{c l}
                  1 & \text{ if } a_{n} \neq 0 , \\
                  0 & \text{ if } a_{n} = 0 .
              \end{array} \right. $$
Then the generating function
$\sum_{n \geq 0} b_{n} \cdot t^{n}$ is rational.
\label{theorem_Skolem_Mahler_Lech}
\end{theorem}

\begin{lemma}
The generating function $K_{x,y}(t)$ is rational.
\end{lemma}
\begin{proof}
Let $G(t)$ denote the matrix
$$ G(t) = M + M^{2} \cdot t + M^{3} \cdot t^{2} + \cdots
      = M \cdot (I - M \cdot t)^{-1}  $$
and let $G_{x,y}(t)$ denote the $(x,y)$ entry of this matrix.
Clearly $G_{x,y}(t)$ is the generating function for
the number of walks from the vertex $x$ to the vertex $y$
where the coefficient of $t^{m-1}$ is the number of walks
of length $m$.
Furthermore, it is clear that $G_{x,y}(t)$ is
a rational function.
Hence by Theorem~\ref{theorem_Skolem_Mahler_Lech}
the result follows.
\end{proof}

When the underlying digraph is strongly connected
and has period $d$, it easy to observe that
$K_{x,y}(t)$ is the rational function $t^{r}/(1-t^{d})$
minus a finite number of terms, where the lengths of the walks
from $x$ to $y$ are congruent to $r$ modulo $d$.

\begin{theorem}
The $\ab$-series $\Psi_{x,y}$ is
a rational generating function in the non-commutative variables
$\av$ and $\bv$.
\label{theorem_ab_rational}
\end{theorem}
\begin{proof}
We first restrict ourselves to 
summing weights of chains
which has length $k+1$
in the level poset,
that is, after excluding
the minimal and maximal element,
those chains consisting of $k$ elements.
The matrix enumerating such
chains is given by the product
$$   K(\av-\bv) \cdot \bv \cdot K(\av-\bv) \cdot \bv \cdots
                \cdot \bv \cdot K(\av-\bv) 
    =
     K(\av-\bv) \cdot (\bv \cdot K(\av-\bv))^{k}  .   $$
Summing over all $k \geq 0$, we obtain
\begin{equation}
   \Psi =   K(\av-\bv) \cdot (I - \bv \cdot K(\av-\bv))^{-1}  .   
\label{equation_Psi}
\end{equation}
Hence each entry of the matrix $\Psi$ is
a rational generating function in $\av$ and $\bv$.
\end{proof}

We turn our attention to the $\cd$-index of level Eulerian posets.
\begin{theorem}
For a level Eulerian poset
the $\ab$-series $\Psi_{x,y}$ is
a rational generating function in the non-commutative variables
$\cv$ and $\dv$.
\label{theorem_cd_rational}
\end{theorem}
We call the resulting generating function
guaranteed in Theorem~\ref{theorem_cd_rational}
the {\em $\cd$-series}.

\begin{proof}[Proof of Theorem~\ref{theorem_cd_rational}]
Observe equation~\eqref{equation_Psi}
is equivalent to
\begin{equation}
   \Psi 
     =
   K(\av-\bv)
     +
   K(\av-\bv) \cdot \bv \cdot \Psi .
\label{equation_Psi_1}
\end{equation}
Consider the involution that
exchanges the variables $\av$ and $\bv$.  Note that this involution
leaves series expressed in $\cv$ and $\dv$ invariant. Apply
this involution to equation~\eqref{equation_Psi_1} gives
\begin{equation}
  \Psi
       =
   K(\bv-\av)  +  K(\bv-\av) \cdot \av \cdot \Psi  .
\label{equation_Psi_2}
\end{equation}
Add the two equations~\eqref{equation_Psi_1}
and~\eqref{equation_Psi_2} and divide by $2$.
\begin{equation}
  \Psi
       =
   (K(\av-\bv) + K(\bv-\av))/2 
       +
   (K(\av-\bv) \cdot \bv  +  K(\bv-\av) \cdot \av)/2 \cdot \Psi  .
\label{equation_Psi_1_2}
\end{equation}
Divide the generating function $K(t)$ into its even, respectively odd,
generating function, that is,
let
$$   K_{0}(t) = \frac{ K(\sqrt{t}) + K(-\sqrt{t}) }{2}
     \:\:\:\: \text{ and } \:\:\:\:
     K_{1}(t) = \frac{ K(\sqrt{t}) - K(-\sqrt{t}) }{2 \cdot \sqrt{t}} .  $$
We have
$K(t) = K_{0}(t^{2}) + K_{1}(t^{2}) \cdot t$
and
$K(-t) = K_{0}(t^{2}) - K_{1}(t^{2}) \cdot t$.
Note that
\begin{eqnarray*}
   (K(\av-\bv) + K(\bv-\av))/2 
  & = &
   K_{0}(\cv^{2} - 2 \cdot \dv) , \\
   (K(\av-\bv) \cdot \bv  +  K(\bv-\av) \cdot \av)/2
  & = &
   K_{0}(\cv^{2} - 2 \cdot \dv) \cdot \cv
    +
   K_{1}(\cv^{2} - 2 \cdot \dv) \cdot (2 \cdot \dv - \cv^{2}) .
\end{eqnarray*}
The result now follows since $K_{0}$ and $K_{1}$ are rational
generating functions and
we can solve for $\Psi$ in
equation~\eqref{equation_Psi_1_2}.
\end{proof}

Bayer and Hetyei~\cite{Bayer_Hetyei_G}
proved that
the $\ab$-index of a half-Eulerian poset
is a polynomial in the two variables $\av$ and $(\av-\bv)^{2}$.
We now show this also holds for the rational
series of a level half-Eulerian poset.
Define the algebra morphism $f_{\leftrightarrow}$ on
$\Rrr\langle\langle \av,\bv \rangle\rangle$
by $f_{\leftrightarrow}(\av-\bv) =\av-\bv$
and $f_{\leftrightarrow}(\bv) = 2\bv$.
It is then easy to observe from
the chain definition~\eqref{equation_chain_definition}
of the $\ab$-index that for any poset~$P$
the $\ab$-index of the poset~$P$ and its horizontal double
are related by
$\Psi(D_{\leftrightarrow}(P)) = f_{\leftrightarrow}(\Psi(P))$.
\begin{corollary}
The $\ab$-series $\Psi_{x,y}$ of a level half-Eulerian poset
is a rational generating function in the non-commutative variables
$\av$ and $(\av-\bv)^{2}$.
\end{corollary}
\begin{proof}
Consider the horizontal double of the level half-Eulerian poset.
Its $\ab$-series is a rational generating function
in terms of $\av+\bv = \cv$ and $(\av-\bv)^{2} = \cv^{2} - 2 \dv$.
The result follows by applying the inverse morphism
$f_{\leftrightarrow}^{-1}$ to this
rational series.
\end{proof}

We similarly define the algebra morphism $f_{\updownarrow}$
by
$f_{\updownarrow}(\av-\bv) = (\av-\bv)^{2}$
and
$f_{\updownarrow}(\bv) = \bv (\av-\bv) + (\av-\bv) \bv + \bv^{2} 
                       = \av\bv + \bv\av + \bv^{2}$.
By the chain definition~\eqref{equation_chain_definition}
of the $\ab$-index we can conclude that
$\Psi(D_{\updownarrow}(P)) = f_{\updownarrow}(\Psi(P))$.
We end this section by
presenting the corresponding results for horizontal-
and vertical-doubling of a level poset.
The proof is straightforward and hence omitted.
\begin{proposition}
Let $P$ a level poset with underlying matrix $M$.
Then we have
$$
\Psi(D_{\leftrightarrow}(P))
   =
\begin{pmatrix}
f_{\leftrightarrow}(\Psi(P)) & f_{\leftrightarrow}(\Psi(P)) \\
f_{\leftrightarrow}(\Psi(P)) & f_{\leftrightarrow}(\Psi(P))
\end{pmatrix}  
$$
and
$$
\Psi(D_{\updownarrow}(P))
   =
\begin{pmatrix}
\av \cdot f_{\updownarrow}(\Psi(P)) &
      I + \av \cdot f_{\updownarrow}(\Psi(P)) \cdot \av \\
f_{\updownarrow}(\Psi(P)) &
      f_{\updownarrow}(\Psi(P)) \cdot \av \\
\end{pmatrix} ,
$$
where the two morphisms
$f_{\leftrightarrow}$ and $f_{\updownarrow}$
are applied entrywise to the matrices.
\end{proposition}

\section{Computing the $\cd$-series}

The recursions~\eqref{equation_Psi_1},
\eqref{equation_Psi_2} and~\eqref{equation_Psi_1_2}
are not very practical for explicitly computing the
$\cd$-series of a level Eulerian poset.
In this section we offer a different method
to show the $\cd$-series has a given expression
based upon the coalgebraic techniques
developed in~\cite{Ehrenborg_Readdy_coproducts}.

Define a derivation
$\Delta : \Rrr\langle\langle \av,\bv \rangle\rangle
   \longrightarrow
         \Rrr\langle\langle \av,\bv,\tv \rangle\rangle$
by
$\Delta(\av) = \Delta(\bv) = \tv$, $\Delta(1) = 0$
and require that it satisfy the product rule
$\Delta(u \cdot v) = \Delta(u) \cdot v + u \cdot \Delta(v)$.
It is straightforward to verify that this derivation is
well-defined. Observe that the coefficient of a monomial $u$
in $\Delta(v)$ is zero unless $u$ contains exactly one $\tv$.

Note that for a formal power series $u$ without constant term
we have that
$$  \Delta\left( \frac{1}{1-u} \right)
   =
        \frac{1}{1-u} \cdot \Delta(u) \cdot \frac{1}{1-u}  , $$
since
$\Delta(u^{m}) = \sum_{i=0}^{m-1} u^{i} \cdot \Delta(u) \cdot u^{m-1-i}$
and then by summing over all $m$.

When restricting the derivation $\Delta$ to non-commutative polynomials
$\Rrr\langle \av,\bv \rangle$, it becomes equivalent to the coproduct
on $\ab$-polynomials introduced by Ehrenborg and Readdy
in~\cite{Ehrenborg_Readdy_coproducts}.
To see this fact, observe that the subspace of
$\Rrr\langle \av,\bv,\tv \rangle$ spanned by monomials
containing exactly one $\tv$ is isomorphic to
$\Rrr\langle \av,\bv \rangle \tensor \Rrr\langle \av,\bv \rangle$
by mapping the variable $\tv$ to the tensor sign, that is,
$u \cdot \tv \cdot v \longmapsto u \tensor v$.

We need two properties of the derivation $\Delta$.
The first is that the $\ab$-index is a coalgebra homomorphism,
that is, for a poset $P$ we have
\begin{equation}
    \Delta(\Psi(P))
  =
    \sum_{\hz < x < \ho}
        \Psi([\hz,x]) \cdot \tv \cdot \Psi([x,\ho])  .
\label{equation_coalgebra_morphism}
\end{equation}
See~\cite[Proposition~3.1]{Ehrenborg_Readdy_coproducts}.
Applying~\eqref{equation_coalgebra_morphism}
to all the rank $m+1$ intervals of a level poset, we have that
$$  \Delta(\Psi_{m})
  =
    \sum_{i=0}^{m-1} \Psi_{i} \cdot \tv \cdot \Psi_{m-1-i} , $$
where $\Psi_{m}$ denotes the degree $m$ terms of
the $\ab$-series $\Psi$.
The second property of the derivation
is that when restricting the derivation
to $\ab$-polynomials of degree $n$, the kernel of the map
is spanned by $(\av-\bv)^{n}$.
See~\cite[Lemma~2.2]{Ehrenborg_Readdy_coproducts}.

We now prove the main result of this section.
It is a method to recognize the $\ab$-series matrix of a level poset.
\begin{theorem}
The $n \times n$ matrix $\Psi$ of the $\ab$-series 
of a level poset
is the unique solution to the equation system
\begin{eqnarray}
  \left. \Psi\right|_{\av = t, \bv = 0}
  & = &
  K(t) ,
\label{equation_Delta_1} \\
\Delta(\Psi)
  & = &
\Psi \cdot \tv \cdot \Psi .
\label{equation_Delta_2}
\end{eqnarray}
\label{theorem_Delta}
\end{theorem}
\begin{proof}
Let $\Gamma$ be a solution to the two
equations~\eqref{equation_Delta_1} and~\eqref{equation_Delta_2}.
Write $\Gamma$ as the sum
$\sum_{m \geq 0} \Gamma_{m}$ where the entries
of the matrix $\Gamma_{m}$ are homogeneous of degree $m$.
By induction on $m$ we will prove that
$\Gamma_{m}$ is equal to $\Psi_{m}$, the $m$th homogeneous
component of the matrix $\Psi$. The base case $m=0$ is as follows.
$$ \Gamma_{0} = \left.\Gamma\right|_{\av=\bv=0}
              = \left.K(t)\right|_{t=0} = M = \Psi_{0} . $$
Now assume the statement is true for all values less than $m$.
Observe that the $m$th component of
equation~\eqref{equation_Delta_2} is
$$ \Delta(\Gamma_{m})
   =
\sum_{i=0}^{m-1}
\Gamma_{i} \cdot \tv \cdot \Gamma_{m-1-i}
   =
\sum_{i=0}^{m-1}
\Psi_{i} \cdot \tv \cdot \Psi_{m-1-i}
   =
\Delta(\Psi_{m}) . $$
Hence the difference $\Gamma_{m} - \Psi_{m}$ is
a constant matrix $N$ times the $\ab$-polynomial $(\av-\bv)^{m}$.
However, the matrix $N$ is zero since
$\left. \Gamma_{m}\right|_{\bv=0} = \left. \Psi_{m}\right|_{\bv=0}$,
proving that $\Gamma_{m}$ is equal to $\Psi_{m}$,
completing the induction.
\end{proof}

\begin{example}
{\rm
Consider the level Eulerian poset in Figure~\ref{figure_main}.
We claim that its $\cd$-series matrix $\Psi$ is given by
$$
   \Psi
  =
   \begin{pmatrix}
   \frac{1}{1 - \cv - \dv} &
   \frac{1}{1 - \cv - \dv} \cdot \cv + 1 &
   \frac{1}{1 - \cv - \dv} &
   \frac{1}{1 - \cv - \dv} - 1
                                \\[2 mm]
   \frac{1}{1 - \cv - \dv} &
   \frac{1}{1 - \cv - \dv} \cdot \cv &
   \frac{1}{1 - \cv - \dv} &
   \frac{1}{1 - \cv - \dv}
                                \\[2 mm]
   \cv \cdot \frac{1}{1 - \cv - \dv} &
   \cv \cdot \frac{1}{1 - \cv - \dv} \cdot \cv + 1 &
   \cv \cdot \frac{1}{1 - \cv - \dv} &
   \cv \cdot \frac{1}{1 - \cv - \dv} + 1
                                \\[2 mm]
   \frac{1}{1 - \cv - \dv} &
   \frac{1}{1 - \cv - \dv} \cdot \cv &
   \frac{1}{1 - \cv - \dv} &
   \frac{1}{1 - \cv - \dv}
         \end{pmatrix}  .
$$
It is straightforward to check
the first condition in Theorem~\ref{theorem_Delta}:
$$
   \left. \Psi\right|_{\av = t, \bv = 0}
  =
   \left. \Psi\right|_{\cv = t, \dv = 0}
  =
   \begin{pmatrix}
   \frac{1}{1 - t} &
   \frac{1}{1 - t} &
   \frac{1}{1 - t} &
   \frac{1}{1 - t} - 1
                                \\[2 mm]
   \frac{1}{1 - t} &
   \frac{1}{1 - t} - 1 &
   \frac{1}{1 - t} &
   \frac{1}{1 - t}
                                \\[2 mm]
   \frac{1}{1 - t} - 1 &
   \frac{1}{1 - t} - t &
   \frac{1}{1 - t} - 1 &
   \frac{1}{1 - t}
                                \\[2 mm]
   \frac{1}{1 - t} &
   \frac{1}{1 - t} - 1 &
   \frac{1}{1 - t} &
   \frac{1}{1 - t}
         \end{pmatrix}
  =
   K(t) .
$$
To verify the second condition, define the four vectors
$$
 x = \begin{pmatrix} 1 \\ 1 \\ \cv \\ 1 \end{pmatrix} ,
   \:\:\:\:
 y = \begin{pmatrix} 0 \\ 0 \\ 1 \\ 0 \end{pmatrix} ,
   \:\:\:\:
 z = \begin{pmatrix} 1 & \cv & 1 & 1 \end{pmatrix} 
   \:\:\:\: \text{ and } \:\:\:\:
 w = \begin{pmatrix} 0 & 1 & 0 & 0 \end{pmatrix} ,
$$
and the matrix
$$
 A = \begin{pmatrix}
          0 &  1 &  0 & -1 \\
          0 &  0 &  0 &  0 \\
          0 &  1 &  0 &  1 \\
          0 &  0 &  0 &  0
     \end{pmatrix} .
$$
We have the following relations between
this matrix and these vectors:
$$
     z \cdot \tv \cdot x = 2 \cdot \tv + \cv \tv + \tv \cv
                         = \Delta(\cv + \dv),
\:\: 
     A \cdot x = 2 \cdot y,
\:\: 
     z \cdot A = 2 \cdot w
\:\: \text{and} \:\: 
     A \cdot \tv \cdot A  = 0    .  $$
Furthermore, the derivative $\Delta$ acts as follows
$$    \Delta(x) = 2 \cdot \tv \cdot y,
\:\:
      \Delta(z) = 2 \cdot \tv \cdot w
\:\: \text{and} \:\: 
      \Delta(A) = 0              .  $$
Observe now that
$$  \Psi = x \cdot \frac{1}{1 - \cv - \dv} \cdot z + A . $$
Hence we have the following calculation
\begin{eqnarray*}
\Psi \cdot \tv \cdot \Psi
  & = &
\left(x \cdot \frac{1}{1 - \cv - \dv} \cdot z + A\right)
 \cdot \tv \cdot
\left(x \cdot \frac{1}{1 - \cv - \dv} \cdot z + A\right) \\
  & = &
x \cdot \frac{1}{1 - \cv - \dv} \cdot z \cdot \tv \cdot
x \cdot \frac{1}{1 - \cv - \dv} \cdot z  \\
  &   &
+
x \cdot \frac{1}{1 - \cv - \dv} \cdot z \cdot A \cdot \tv
+
\tv \cdot A \cdot x \cdot \frac{1}{1 - \cv - \dv} \cdot z
+
A \cdot \tv \cdot A \\
  & = &
x \cdot \frac{1}{1 - \cv - \dv} \cdot \Delta(\cv + \dv)
  \cdot \frac{1}{1 - \cv - \dv} \cdot z  \\
  &   &
+
x \cdot \frac{1}{1 - \cv - \dv} \cdot 2 \cdot w \cdot \tv
+
\tv \cdot 2 \cdot y \cdot \frac{1}{1 - \cv - \dv} \cdot z \\
  & = &
x \cdot \Delta\left(\frac{1}{1 - \cv - \dv}\right) \cdot z
+
x \cdot \frac{1}{1 - \cv - \dv} \cdot \Delta(z)
+
\Delta(x) \cdot \frac{1}{1 - \cv - \dv} \cdot z \\
  & = &
\Delta\left(x \cdot \frac{1}{1 - \cv - \dv} \cdot z\right) \\
  & = &
\Delta\left(\Psi\right) ,
\end{eqnarray*}
proving our claim.
}
\end{example}

As a corollary to this example we obtain
an interesting Eulerian poset
whose $\cd$-index has all of its coefficients to be $1$.
\begin{corollary}
The $\cd$-index of the interval
$[(1,0),(1,m+1)]$ in the level poset 
in Figure~\ref{figure_main}
is the sum of all $\cd$-monomials of degree $m$.
\label{corollary_sum_of_monomials}
\end{corollary}

\section{Concluding remarks}

Given a non-commutative rational formal power series in the variables
$\av$ and $\bv$ which can be expressed in terms of $\cv$ and
$\dv$, is it necessarily a non-commutative rational formal power series
in the variables $\cv$ and $\dv$? In other words, is the
following equality true
$$   F_{\text{rat}}\langle\langle \av,\bv \rangle\rangle
          \cap
     F\langle\langle \cv,\dv \rangle\rangle
          =
     F_{\text{rat}}\langle\langle \cv,\dv \rangle\rangle   ,  $$
where $F$ is a field?
It is clear that right-hand side of the above
is contained in the left-hand side.

Corollary~\ref{corollary_sum_of_monomials}
suggests a question about the existence of Eulerian posets.
For which subsets $M$ of $\cd$-monomials of degree $m$
is there an Eulerian poset whose $\cd$-index is the sum
of the monomials in~$M$?
The two extreme cases 
$\cv^{m}$ and $\sum_{\deg(w) = m} w$
both arrive from level Eulerian posets.

An open question is if the eigenvalues or
other classical matrix invariants carry information
about the corresponding level poset,
such as if the level poset is Eulerian or shellable.

\section*{Acknowledgments}

The first author was partially funded by National Science Foundation grant
DMS-0902063.
The authors thank the Department of Mathematics
at the University of Kentucky for funding a research visit 
for the second author to the University of Kentucky,
where part of this research was carried out.

\newcommand{\journal}[6]{{\sc #1,} #2, {\it #3} {\bf #4} (#5), #6.}
\newcommand{\book}[4]{{\sc #1,} ``#2,'' #3, #4.}

\end{document}